\documentclass[smallextended,numbook,runningheads]{svjour3}     
\smartqed  
\usepackage{xcolor}
\usepackage{graphicx}
\usepackage{amsmath}
\usepackage{amsfonts}
\usepackage{epstopdf} 
\usepackage{hyperref}
\usepackage{mathptmx}      
\journalname{BIT}

\begin{document}

\title{Riemannian Inexact Newton Method for Structured Inverse Eigenvalue and Singular Value Problems
}
%
%

\author{
Chun-Yueh Chiang\thanks {Center for General Education, National Formosa
University, Huwei 632, Taiwan (\texttt{chiang@nfu.edu.tw}).
This
research was supported in part  by the Ministry of Science and Technology of Taiwan under grant 105-2115-M-150-001.}
\and
Matthew M. Lin\thanks{Corresponding author. Department of Mathematics, National Cheng Kung University, Tainan 701, Taiwan (\texttt{mhlin@mail.ncku.edu.tw}).
This
research was supported in part by the Ministry of Science and Technology of Taiwan
under grant 107-2115-M-006 -007 -MY2.}
\and
Xiao-Qing Jin\thanks
{Department of Mathematics, University of Macau, Macao, China (\texttt{xqjin@umac.mo}).
This research was supported in part by the research grant MYRG2016-00077-FST from University of Macau.}
}

\date{Received: date / Accepted: date}
\maketitle

\begin{abstract}
Inverse eigenvalue and singular value problems have been widely discussed for decades. The well-known result is the Weyl-Horn condition, which presents the relations between the eigenvalues and  singular values of an arbitrary matrix.
This result by Weyl-Horn then leads to an interesting inverse problem, i.e., how to construct a matrix with desired eigenvalues and singular values.
In this work, we do that and more. We propose an eclectic mix of techniques from differential geometry and the inexact Newton method for solving inverse eigenvalue and singular value problems as well as additional desired characteristics such as nonnegative entries, prescribed diagonal entries, and even predetermined entries.
We show theoretically that our method converges globally and quadratically, and we provide numerical examples to demonstrate the robustness and accuracy of our proposed method. 
{Having theoretical interest, we provide in the appendix a necessary and sufficient condition for the existence of a $2\times 2$ real matrix, or even a nonnegative matrix, with prescribed eigenvalues, singular values, and main diagonal entries. 
}

\end{abstract}



\titlerunning{Riemannian INM for structured IESPs}


\section{Introduction}

Let $|\lambda_1|\geq |\lambda_2|\geq \cdots\geq |\lambda_n| \geq 0$ and $\sigma_1 \geq \cdots\geq \sigma_n \geq 0$ be the eigenvalues and singular values of a given $n\times n$ matrix $A$.   In~\cite{Weyl1949} Weyl showed that sets of eigenvalues and singular values satisfy the following necessary condition:
\begin{subequations}\label{eq:Weyl}
\begin{eqnarray}
&&\prod_{j=1}^k|\lambda_j| \leq \prod_{j=1}^k \sigma_j, \quad k = 1,\ldots, n-1,\\
&&\prod_{j=1}^n |\lambda_j| = \prod_{j=1}^n \sigma_j.
\end{eqnarray}
\end{subequations}
Moreover, Horn~\cite{Horn1954b} proved that condition~\eqref{eq:Weyl}, called the Weyl-Horn condition, is also sufficient for constructing triangular matrices with prescribed eigenvalues and singular values.  Research interest in inverse eigenvalue and singular value problems can be tracked back to the open problem raised by Higham
 in~\cite[Problem 26.3]{Higham1996}, as follows:
 \vskip .1in
\begin{minipage}{0.9\textwidth}
 Develop an efficient algorithm for computing a unit upper triangular $n\times n$ matrix with the prescribed singular values $\sigma_1,\ldots,\sigma_n$, where
$\prod_{j=1}^n \sigma_j = 1$.
\end{minipage}
\vskip .1in
%
\noindent This problem, which was solved by Kosowski and Smoktunowicz~\cite{Kosowski2000}, leads to the following interesting inverse eigenvalue and singular value problem (IESP):
 \vskip .1in
\begin{minipage}{0.9\textwidth}
 (\textbf{IESP}) Given two sets of numbers  ${\boldsymbol{\lambda}}
 = \{\lambda_1, \ldots, \lambda_n\}$ and $\boldsymbol{\sigma}
 = \{\sigma_1, \ldots,\sigma_n\}$
satisfying~\eqref{eq:Weyl}, find a real $n\times n$ matrix with eigenvalues ${\boldsymbol{\lambda}}$ and singular values ${\boldsymbol{\sigma}}$.
\end{minipage}
\vskip .1in

\noindent

The following factors make the IESP difficult to solve:
 \begin{itemize}
 \item Often the desired matrices are real. This problem was solved by the authors of~\cite{Chu2000} with prescribed real {eigenvalues} and singular values. {The method for finding a general real-valued matrix with prescribed complex-conjugate eigenvalues and singular values was also investigated in~\cite{Li2001}. In this work, we take an alternative approach to tackle this problem and add further constraints.}

 \item Often the desired matrices are structured. Corresponding to physical applications, the recovered matrices often preserve some common structure such as nonnegative entries or predetermined diagonal entries~\cite{Chu1999,Chu2017}. {In this paper, specifically, we offer the condition of the existence of a nonnegative matrix provided that eigenvalues, singular values, and diagonal entries are given.}  Furthermore, solving the IESP with respect to the diagonal constraint is not enough because entries of the recovered matrices should preserve certain patterns, for example, non-negativity, which correspond to original observations. How to tackle this structured problem is the main thrust of this paper.

 \end{itemize}

{The IESP can be regarded as a natural generalization of the inverse eigenvalue problems, which is known for its a wide variety of applications such as the pole assignment problem~\cite{EKChu1996,Li1997,Datta2006}, applied mechanics~\cite{Gohberg82,Datta00,Nichols01,Datta02,Chu07b}, and inverse Sturm-Liouville problem~\cite{Golub73,Golub87,Gladwell04,Moller2015}. Thus applications of the IESP could be found in wireless communication~\cite{Rao2005,Cotae2006,Tropp2004} and quantum  information science~\cite{Deakin1992,Jacek2009,Chu2017}. }
Research results advanced thus far for the IESP do not fully address the above scenarios. Often, given a set of data, the IESP is studied in parts. That is, there have been extensive investigations of the conditions for the existence of a matrix when the singular values and eigenvalues are provided (i.e., the Weyl-Horn condition~\cite{Weyl1949,Horn1954b}),  when the singular values and main diagonal entries are provided (i.e., the Sing-Thompson condition~\cite{Sing1976,Thompson1977}), or when the eigenvalues and main diagonal entries are provided (i.e., the Mirsky condition~\cite{Mirsky1958}). {Also, the above conditions have given rise to numerical approaches, as found in~\cite{Chan1983,Chu1995,Chu1999,Chu2000,Dhillon2005,Kosowski2000,Zha1995}. 

Our significance in this work is to consider these conditions together.}
One relatively close result is given in~\cite{Chu2017}, where the authors consider a new type of IESP that requires that all three constraints, i.e., eigenvalues, singular values, and diagonal entries, be satisfied simultaneously.
Theoretically, Wu and Chu generalize the classical Mirsky, Sing-Thompson, and Weyl-Horn conditions and provide one sufficient condition for the existence of a matrix with prescribed eigenvalues,  singular values, and diagonal entries
when $n\geq 3$.
Numerically, Wu and Chu establish a dynamic system for  constructing such a matrix, in which real eigenvalues are given.
In this work, we solve an IESP with complex conjugate eigenvalues and with entries fixed at certain locations.
Also, we provide the necessary and sufficient condition of the existence of a $2\times 2$ nonnegative matrix with prescribed eigenvalues,  singular values, and  diagonal elements.  Note that, in general, the solution of the IESP is not unique or difficult to find once structured requirements are added.
To solve an IESP with some specific feature, we combine techniques from differential geometry and for solving nonlinear equations.

%
%
%
%
%
%

We organize this paper as follows.
In section 2, we propose the use of the Riemannian inexact Newton method for solving an IESP with complex conjugate eigenvalues. In section 3, we show that the convergence is quadratic. In section 4, we demonstrate the application of our technique to an IESP with a specific structure that includes nonnegative or predetermined entries 
to show the robustness and efficiency of our proposed approaches. 
{The concluding remarks and the solvability of the IESP of a $2\times 2$ matrix are given in section 5 and the appendix, respectively.}

\section{Riemannian inexact Newton method}\label{sec_newton}
In this section, we explain how the Riemannian inexact Newton method can be applied to the IESP. The problem of optimizing a function on a matrix manifold has received much attention in the {scientific} and engineering fields due to its peculiarity and capacity.
Its applications include, but are not limited to, the study of eigenvalue problems~\cite{Chu1990,Chu1991,Chu1992,Absil2008,Alexandrov2006,Chu05,Chu2008,Zhao2015,Zhao2016a,Zhao2016b,Chu2017,Zhao2017}, matrix low rank approximation~\cite{Boyd2005,Igor2007}, and nonlinear matrix equations~\cite{Chu2015,Chu2016}. Numerical methods for solving problems involving matrix manifolds rely on interdisciplinary
inputs from differential geometry, optimization theory, and  gradient flows.

To begin, let $\mathcal{O}(n) \subset \mathbb{R}^{n\times n}$  be the group of $n\times n$ real orthogonal matrices, and
let
${\boldsymbol{\lambda}}
 = \{\lambda_1, \ldots, \lambda_n\}$
 and $\boldsymbol{\sigma}
 = \{\sigma_1, \ldots,\sigma_n\}$ be the eigenvalues
 and singular values of an $n\times n$ matrix. We assume without loss of generality that:
 \begin{equation*}
 \lambda_{2i-1} = \alpha_i + \beta_i \sqrt{-1}, \quad
 \lambda_{2i} = \alpha_i - \beta_i \sqrt{-1}, \quad
 i = 1,\ldots, k; \quad \lambda_i\in \mathbb{R}, \quad i=2k+1,\ldots,n,
 \end{equation*}
 where $\alpha_i, \beta_i \in\mathbb{R}$ with $\beta_i\neq 0$ for $i = 1,\ldots, k$, and we define the corresponding block diagonal matrix
\begin{align*}
\begin{array}{ccl}
 \Lambda
&\!\!=\!\!&\mathrm{diag}\left \{
\left[
   \begin{array}{rr}
     \alpha_{1}& \beta_{1} \\
     -\beta_{1} &  \alpha_{1}\\
   \end{array}
 \right],\ldots, \left[
   \begin{array}{rr}
     \alpha_{k} & \beta_{k} \\
     -\beta_{k} & \alpha_{k}\\
   \end{array}
 \right], {\lambda}_{2k+1},\ldots, {\lambda}_{2n} \right \}\label{form_H}
\end{array}
\end{align*}
 and the diagonal matrix
 \begin{align*}
\begin{array}{ccl}
 \Sigma
&\!\!=\!\!&\mathrm{diag}\left \{
    {\sigma}_{1},\ldots, {\sigma}_{n} \right \}.
\end{array}
\end{align*}

  Then the IESP is equivalent to
 finding matrices
 $U,\, V,\, Q\in \mathcal{O}(n)$,  and
 $$W\in\mathcal{W}(n)
:= \{W\in\mathbb{R}^{n\times n} \,|\,
W_{i,j} = 0 \mbox{ if }
\Lambda_{i,j} \neq 0
 \mbox{ or }
i\geq j, \mbox{ for }\, 1\leq i,j \leq n
\},$$
which satisfy the following equation:
 \begin{equation}\label{eq:IESP0}
 F(U,V,Q,W) =U\Sigma V^\top - Q(\Lambda+W)Q^\top = \mathbf{0}.
 \end{equation}
 Here, we may assume without loss of generality that $Q$ is an identity matrix and simplify Eq.~\eqref{eq:IESP0} as follows:
 \begin{equation}\label{eq:IESP}
 F(U,V,W) =U\Sigma V^\top - (\Lambda+W) = \mathbf{0}.
 \end{equation}

Let $X = (U,V,W) \in
\mathcal{O}(n)
 \times \mathcal{O}(n) \times \mathcal{W}(n)$.  Upon using Eq.~\eqref{eq:IESP}, we can see that we might solve the IESP by
\begin{equation}\label{eq:fx}
\mbox{finding }X \in
\mathcal{O}(n)
 \times \mathcal{O}(n) \times \mathcal{W}(n)
\mbox{ such that } F(X) = \mathbf{0},
\end{equation}
where $F:
\mathcal{O}(n)  \times \mathcal{O}(n) \times \mathcal{W}(n)
\rightarrow \mathbb{R}^{n\times n}
$ is continuously differentiable.  By making an initial guess, $X_0$, one immediate way to solve Eq.~\eqref{eq:fx} is to apply the Newton method and generate a sequence of iterates by
solving
\begin{eqnarray}\label{eq:nw1}
DF(X_{k}) [\Delta X_k] = -F(X_k),
\end{eqnarray}
for $\Delta X_k \in T_{X_k}(\mathcal{O}(n) \times
 \mathcal{O}(n) \times \mathcal{W}(n))$ and set
\begin{eqnarray*}
X_{k+1} = R_{X_k} (\Delta X_k),
\end{eqnarray*}
where $DF(X_k)$ represents the differential of $F$ at $X_k$ and $R$ is a retraction on $\mathcal{O}(n) \times \mathcal{O}(n) \times \mathcal{W}(n)$. Since Eq.~\eqref{eq:nw1} is an underdetermined system,  it may have more than one solution.
Let $DF(X_k)^*$ be the adjoint operator of $DF(X_k)$.
 In our calculation, we choose the solution $\Delta X_k$ with the minimum norm by letting~\cite[Chap.  6]{Luenberger1969}
\begin{equation}\label{eq:NewtonStep}
\Delta X_k =DF(X_k)^* [\Delta Z_k],
\end{equation}
where $\Delta Z_k \in T_{F(X_k)} (\mathbb{R}^{n\times n})$ is a solution for
\begin{equation}\label{eq:conjugate}
\left (DF(X_k)\circ DF(X_k)^* \right )[\Delta Z_k] = -F(X_k).
\end{equation}
Note that {the notation $\circ$ represents the composition of two operators $DF(X_k)$ and $DF(X_k)^*$.} This implies that
the operator $DF(X_k)\circ DF(X_k)^*$ is symmetric and positive semidefinite. If, as is the general case, the operator 
{
$DF(X_k)\circ DF(X_k)^*:
T_{F(X_k)} (\mathbb{R}^{n\times n}) \rightarrow \mathbb{R}^{n\times n}
$
}
is invertible, we can compute the optimal solution in~\eqref{eq:NewtonStep}.

Note that solving for the root of Eq.~\eqref{eq:conjugate} could  be unnecessary  and computationally {time-consuming}, and that the linear model given by Eq.~\eqref{eq:conjugate} is large-scale or the resulting iteration $X_k$ is far from the root of condition~\eqref{eq:fx}~\cite{Simonis2006}.  By analogy with the classical Newton method~\cite{Eisenstat1994}, we adopt the ``inexact" Newton method  on Riemannian manifolds, i.e., without solving Eq.~\eqref{eq:conjugate}
exactly, we  repeatedly apply the
conjugate gradient
(CG) method to find $\Delta Z_k\in T_{F(X_k)} (\mathbb{R}^{n\times n})$, such that:
\begin{equation}\label{eq:inm}
\|(DF(X_k)\circ DF(X_k)^*) [\Delta Z_k] + F(X_k)\| \leq \eta_k \|F(X_k)\|,
\end{equation}
for some constant $\eta_k \in [0,1)$, is satisfied. Then, we update $X_k$ corresponding to $\Delta Z_k$ until the stopping criterion is satisfied. Here, the notation $\|\cdot\|$ is the Frobenius norm.
Note that in our calculation, {the elements} in the product space $\mathbb{R}^{n\times n} 
\times \mathbb{R}^{n\times n}\times \mathbb{R}^{n\times n}$ are computed using the standard Frobenius inner product:
\begin{equation}\label{eq:InNorm}
\left< (A_1, A_2,A_3),
(B_1, B_2,B_3)
\right>_{F}
:=  \left< A_1,B_1
\right> + \left< A_2,B_2
\right> + \left< A_3,B_3
\right>,
\end{equation}
where $\left< A,B
\right> := \mbox{trace}(AB^\top) $ for any $A, B\in\mathbb{R}^{n\times n}$ {and the induced norm $\|X\|_{F}
= \sqrt{\left<X,X\right>_{F}
}$
(or, simply, $\left<X,X\right>$ and
$\|X\|$ without the risk of confusion)
for any $X
\in \mathbb{R}^{n\times n} 
\times \mathbb{R}^{n\times n}\times \mathbb{R}^{n\times n}$.}

%

Then, the linear mapping $DF(X_k)$
at $\Delta X_k = (\Delta U_k,\Delta V_k, 
\Delta W_k)
   \in T_{X_k}(\mathcal{O}(n)
   \times \mathcal{O}(n) \times \mathcal{W}(n))$
is given by:
%
%
\begin{equation*} 
DF(X_k)[\Delta X_k]  =
 \Delta U_k \Sigma V_k^\top
   + U_k\Sigma \Delta V_k^\top
   -  \Delta W_k.
\end{equation*}
Let
$DF(X_k)^*: T_{F(X_k)} (\mathbb{R}^{n\times n}) \rightarrow
T_{X_k}(\mathcal{O}(n)
 \times \mathcal{O}(n) \times \mathcal{W}(n))
$
be the adjoint
of the mapping $DF(X_k)$.
%
%
  The adjoint $DF(X_k)^*$
is determined by the following:
\begin{equation*}
\left<
\Delta Z_k, DF(X_k)[\Delta X_k]
\right>  = \left<
DF(X_k)^*[\Delta Z_k], \Delta X_k
\right>
\end{equation*}
and can be expressed as follows:
%
\begin{equation*}
DF(X_k)^*[\Delta Z_k]
{=(\Delta U_k,\Delta V_k,\Delta W_k),}
\end{equation*}
where
\begin{eqnarray*}
\Delta U_k &=&
\frac{1}{2}(\Delta Z_k V_k \Sigma^\top
-U_k \Sigma V_k^\top \Delta Z_k^\top U_k
),
\\
\Delta V_k &=&
%
\frac{1}{2}(
\Delta Z_k^\top U_k\Sigma
-V_k
\Sigma^\top U_k^\top
 \Delta Z_k V_k
),
\\
\Delta W_k &=&
%
-H\odot
\Delta Z_k,
\end{eqnarray*}
with the notation $\odot$ representing the Hadamard product (see~\cite{Chu1990,Zhao2017} for a similar discussion).

%

%
%
%
%
%

%

There is definitely no guarantee that the application of the inexact Newton method can achieve a sufficient decrease in the size of the nonlinear residual $\|F(X_k)\|$. This provides motivation for  deriving an iterate for which the size of the nonlinear residual is decreased. One way to do this is to update the Newton step $\Delta X_k$ obtained from Eq.~\eqref{eq:NewtonStep}
by choosing $\theta \in [\theta_{\mathrm{min}}, \theta_{\mathrm{max}}]$, with $0<\theta_{\mathrm{min}}<\theta_{\mathrm{max}}<1$, and setting
\begin{equation}\label{eq:eta1}
\widehat{\Delta X}_k = \Delta X_k, \quad
\hat{{\eta}}_k = \frac{\|F(X_k) + DF(X_k) {\Delta X}_k\|}{\|F(X_k)\|}, 
\end{equation}
and $\eta_k = \hat{\eta}_k$.
Then, we update
\begin{align}\label{eq:update}
& \eta_k \leftarrow 1-\theta(1-\eta_k) \mbox{ and }
 \Delta X_k \leftarrow
 \frac{1-\eta_k}{1-\hat{\eta}_k} \widehat{\Delta X}_k,
\end{align}
while
\begin{equation*}
\|F(X_k)\| - \|F(R_{X_k}(\Delta X_k))\|  > t(1-{\eta}_k)\|F(X_k)\|,
\end{equation*}
or, equivalently,
\begin{equation}\label{eq:line}
 \|F(R_{X_k}(\Delta X_k))\|  < [1-t(1-{\eta}_k)]\|F(X_k)\|,
\end{equation}
for some $t \in[0,1)$~\cite{Eisenstat1994}. Let $qf(\cdot)$ denote the mapping that sends a matrix to the $Q$ factor of its $QR$ decomposition with its
$R$ factor having strictly positive diagonal elements~\cite[Example 4.1.3]{Absil2008}.
Then,  for all $(\xi_U, \xi_V,\xi_W) \in T_{(U,V,W)}  \left(\mathcal{O}(n) \times \mathcal{O}(n) \times \mathcal{W}(n)\right)$,
we can compute the retraction $R$ using the following formula:
\begin{equation*}
R_{(U,V,W)} (\xi_U, \xi_V,\xi_W) = (R_U(\xi_U),R_V(\xi_V),
R_W(\xi_W)),
\end{equation*}
where
\begin{equation*}
R_U(\xi_U) = qf(U+\xi_U), \quad
R_V(\xi_V) = qf(V+\xi_V), \quad
R_W(\xi_W) = W+\xi_W.
\end{equation*}
We call this the Riemannian inexact Newton backtracking method (RINB) and formalize this method in Algorithm~1. To choose the parameter $\theta \in [\theta_{\mathrm{min}}, \theta_{\mathrm{max}}]$, we apply a two-point parabolic model~\cite{Kelley1995,Zhao2017} to achieve a sufficient decrease among steps {6 to 9}. That is, we use the iteration history to model an approximate minimizer of the following scalar function:
\begin{equation*}
f(\lambda) := \|F(R_{X_k} (\lambda \Delta X_k))\|^2
\end{equation*}
by defining a parabolic model, as follows:
\begin{equation*}
p(\lambda) = f(0) + f'(0)\lambda  +   (f(1)-f(0)-f'(0)) \lambda^2,
\end{equation*}
where $f(0) = \|F(X_k)\|^2$, $f'(0) = 2\left< DF(X_k) [\Delta X_k], F(X_k)
\right >
$, and $f(1) = \|F(R_{X_k}(\Delta X_k))\|^2$. 
%
%

{From~\eqref{eq:inm}, it can be shown that the function evaluation
$f'(0)$ should be negative.}
Since $f'(0) < 0$, if $
p''(\lambda) = 2(f(1) - f(0) - f'(0)) > 0$,
then $p(\lambda)$ has its minimum at:
\[
\theta = \frac{-f'(0)}{2(f(1) - f(0) - f'(0))} > 0;
\]
otherwise, if $p''(\lambda) < 0$, we choose $\theta = \theta_{\mathrm{max}}$.
By incorporating two types of selection, we can choose the following:
\begin{equation*}
\theta = {\min}\left\{
{\max}\left\{
\theta_{{\min}},  \frac{-f'(0)}{2(f(1) - f(0) - f'(0))}\right\}, \theta_{{\max}}
\right\}.
\end{equation*}
as the parameter $\theta$ in Algorithm~1~\cite{Kelley1995,Zhao2017}.  In the next section, we mathematically investigate the convergence analysis of Algorithm~1.
%
%
%
%
%


%

\begin{table}[htp]
\begin{center}
\begin{tabular}{l}
\hline
Algorithm 1:  The {Riemannian} inexact Newton backtracking method  \hfill $[X] = \textsc{RINB}( \boldsymbol{\sigma},  \boldsymbol{X_0})$\\
\hline
\textbf{Input:} An initial value $X_0$\\
\textbf{Output:} A numerical solution $X$ {satisfying} $F(X) = \mathbf{0}$\\
 \phantom{1}1\, \textbf{begin}\\
\phantom{1}2 \quad\, Let $\eta_{\mathrm{max}} \in [0.1)$, $\eta_0 =
\mathrm{min}\{
\eta_{\mathrm{max}}, \|F(X_0)\|\}$, and
$t\in[0,1)$, and $0<\theta_{\mathrm{min}}<\theta_{\mathrm{max}}<1$ be given.\\
\phantom{1}3 \quad\, \textbf{repeat}\\
\phantom{1}4 \quad\, \quad\,{{Determine ${\Delta Z_k}$ by using the CG method to~\eqref{eq:conjugate} until~\eqref{eq:inm} holds.}}\\
\phantom{1}5 \quad\, \quad\,{Set $ {\Delta X}_k  = (DF(X_k))^* {\Delta Z_k}$, $\hat{\eta}_k = \frac{\|F(X_k) + DF(X_k) {\Delta X}_k\|}{\|F(X_k)\|}
$, $ \widehat{\Delta X}_k = \Delta X_k$, and ${\eta}_k = \hat{\eta}_k$.
}\\
\phantom{1}6 \quad\, \quad\,\textbf{repeat}\\
\phantom{1}7 \quad\, \quad\, \quad\, Choose $\theta \in [\theta_{\mathrm{min}},
\theta_{\mathrm{max}}
]$.\\
\phantom{1}8 \quad\, \quad\, \quad\,  Update
 $\eta_k \leftarrow 1-\theta(1-\eta_k)$
and ${\Delta X}_k \leftarrow
 \frac{1-\eta_k}{1-\hat{\eta}_k}
\widehat{\Delta X}_k$.\\
\phantom{1}9  \quad\, \quad\,\textbf{until} \eqref{eq:line} holds;\\
10\phantom{1}  \quad  \quad Set $X_{k+1} = R_{X_k} (\Delta X_k)$ and
$\eta_{k+1} = \mathrm{min}\left\{\eta_k,\eta_{\mathrm{max}}, {\|F(X_{k+1})\|} \right\}$.\\
11\phantom{1}  \quad  \quad Replace $k$ by $k+1$.\\
12\phantom{1} \quad \textbf{until} ${\|F(X_{k})\|} < \epsilon $;\\
13\phantom{1} \quad $X = X_k$.\\
14\phantom{1} \textbf{end}\\
\hline
\end{tabular}
\end{center}
\label{default}
\end{table}%

\section{Convergence Analysis}
By combining the classical inexact Newton method~\cite{Eisenstat1994} with optimization techniques on matrix manifolds, Algorithm~1 provides a way to solve the IESP. However, we have yet to theoretically discuss the convergence analysis of Algorithm~1. In this section, we
provide a theoretical foundation for the RINB method, and show that this RINB method converges globally and finally converges quadratically when Algorithm~1 does not terminate prematurely. We address this phenomenon in the following:

\begin{lemma}
Algorithm~1 does not break down at some $X_k$ if and only if $F(X_k)\neq \mathbf{0}$ and the inverse of $DF(X_k)\circ DF(X_k)^*$ exists.
\end{lemma}


%

Next, we provide an upper bound for the approximate solution $\widehat{\Delta X}_k$ in Algorithm~1.
\begin{theorem}\label{thm:stepsize}
Let $\Delta{Z}_k\in T_{F(X_k)} (\mathbb{R}^{n\times n})$ be a solution that satisfies condition~\eqref{eq:inm}
and $$\widehat{\Delta X}_k = DF(X_k)^* [\Delta{Z}_k].$$ Then,
\begin{subequations}\label{eq:length0}
\begin{align}
&\mathrm{(a)}\, \| \widehat{\Delta X}_k
\| \leq (1+ {\hat{\eta}_k} )\|DF(X_k)^\dagger\|
\|F(X_k)\|,\label{eq:steplength}\\
&\mathrm{(b)}\,
\|\sigma_k(\eta)\| \leq
 \frac{ 1+\eta_{\mathrm{max}} }{1-{\eta}_{\mathrm{max}}} (1-\eta) \|
DF(X_k)^\dagger
\|_d
 \|F(X_k)\|,\label{eq:length}
\end{align}
\end{subequations}
where ${\hat{\eta}_k}$ is defined in Eq.~\eqref{eq:eta1}, and
$\sigma_k$ is the backtracking curve used in Algorithm~1, which is defined by the following:
\begin{equation*}
\sigma_k(\eta) = \frac{1-\eta}{1-\hat{\eta}_k}  \widehat{\Delta X}_k 
\end{equation*}
with $\hat{\eta}_k\leq \eta \leq 1$,
 and
 \begin{equation*}
 \|DF(X_k)^\dagger\|
  :=  \max\limits_{\|\Delta Z\| = 1}
{\|DF(X_k)^\dagger[\Delta Z]\|}
  \end{equation*}
 represents the norm of the pseudoinverse of $DF(X_k)$.

\end{theorem}

\begin{proof}
Let $r_k = (DF(X_k)\circ DF(X_k)^*) [\Delta {Z}_k] + F(X_k)$.  We see that
\begin{eqnarray*}
\| \widehat{\Delta X}_k
\| &\leq& \|DF(X_k)^* \circ [DF(X_k)\circ DF(X_k)^*]^{-1}\|
\|r_k - F(X_k)\|\\
&\leq& (1+ \hat{\eta}_k)\|DF(X_k)^\dagger\|
\|F(X_k)\|
\end{eqnarray*}
and
\begin{eqnarray*}
\|\sigma_k(\eta)\| & =  &  \frac{1-\eta}{1-\hat{\eta}_k}
\|
DF(X_k)^\dagger  (r_k-F(X_k))
\| \leq  \frac{ 1+\hat{\eta}_k }{1-\hat{\eta}_k} (1-\eta)
\|
DF(X_k)^\dagger
\|
\|
 F(X_k)
\|
 \\
&\leq&
  \frac{ 1+\eta_{\mathrm{max}} }{1-{\eta}_{\mathrm{max}}} (1-\eta) \|
DF(X_k)^\dagger
\|
\|
 F(X_k)
\|. 
\end{eqnarray*} \qed
\end{proof}

{In our subsequent discussion, we assume that Algorithm~1 does not break down and there is a unique limit point $X_*$ of $\{X_k\}$}.
Since $F$ is continuously differentiable, we have the following:
\begin{equation}\label{eq:dagger}
\|DF(X)^\dagger \| \leq 2 \|DF(X_*)^\dagger \|
\end{equation}
whenever $X \in B_\delta (X_*)$ for a sufficiently small constant $\delta > 0$. Here, the notation $B_\delta (X_*)$ represents a neighborhood of $X_*$ consisting of all points $X$ such that $
\|X - X_*\|
 < \delta$.
By condition~\eqref{eq:length0}, we can show without any difficulty that whenever $X_k$ is sufficiently close to $X_*$,
\begin{eqnarray}
 \| \widehat{\Delta X}_k
\| &\leq&(1+\eta_{\mathrm{max}} )\|DF(X_*)^\dagger\|
\|{F(X_k)}\|,\label{eq:steplengthnew}\\
\|\sigma_k(\eta)\| &\leq&
\Gamma (1-\eta)
 \|{F(X_k)}\|, \quad \hat{\eta}_k\leq \eta \leq 1, \nonumber
\end{eqnarray}
%
where $\Gamma$ is a constant independent of $k$ defined by
\begin{equation*}
\Gamma =  2\frac{ 1+\eta_{\mathrm{max}} }{1-{\eta}_{
\mathrm{
max}}}  \|
DF(X_*)^\dagger
\|.
\end{equation*}

New, we show that the sequence of $\{F(X_k)\}$ eventually converges to zero.

%
\begin{theorem}\label{lem:cov}
Assume that Algorithm~1 does not break down.
If $\{X_k\}$  is the sequence generated in Algorithm~1,
%
then
\begin{equation*}
\lim_{k\rightarrow \infty} F(X_k) = \mathbf{0}.
\end{equation*}

\end{theorem}

\begin{proof}
Observe that
\begin{eqnarray*}
\|F(X_k)\| &= & \|F(R_{X_{k-1}}(\Delta X_{k-1}))\|\leq
(1-t(1-\eta_{k-1})) \|F(X_{k-1})\|\\
 &\leq&
\|F(X_0)\| \prod\limits_{j=0}^{k-1}(1-t(1-\eta_j))
\leq  \|F(X_0)\| e^{-t\sum\limits_{j=0}^{k-1}(1-\eta_j)}.
\end{eqnarray*}
Since $t>0$ and $\lim\limits_{k\rightarrow \infty}\sum\limits_{j=0}^{k-1}(1-\eta_j) = \infty$, we have $\lim\limits_{k\rightarrow \infty} F(X_k) = \mathbf{0}$.
\qed%
\end{proof}

In our iteration, we implement the repeat loop among steps {6 to 9} by selecting a sequence $\{\theta_j\}$, with $\theta_j \in[\theta_{\mathrm{min}},\theta_{\mathrm{max}}]$. For each loop, correspondingly, we let $\eta_k^{(1)} = \hat{\eta}_k$ and $\Delta X^{(1)}
= \widehat{\Delta X}_k
$, and for $j = 2,\ldots,$ we let
\begin{eqnarray}
\eta_k^{(j)} &=& 1-\theta_{j-1}(1-\eta_k^{(j-1)}),
\nonumber\\
\Delta X_k^{(j)} &=& \frac{1-\eta_k^{(j)}}{1-
\hat{\eta}_k}\widehat{\Delta X}_k. \label{eq:Deltakj}
\end{eqnarray}
By induction, then, we can easily show that:
\begin{eqnarray*}
\Delta X_k^{(j)} &=&
\Theta_{j-1} \widehat{\Delta X}_k,
\quad 1-\eta_k^{(j)} =
\Theta_{j-1}  (1-\hat{\eta}_k),
\end{eqnarray*}
where
\begin{equation}\label{Theta}
\Theta_{j-1} = \prod\limits_{\ell=1}^{j-1} \theta_\ell,   \quad j \geq 2.
\end{equation}
That is, the sequence $\{\Delta X_k^{(j)} \}_j$ is a strictly decreasing sequence satisfying
$\lim\limits_{j\rightarrow \infty} \Delta X_k^{(j)} = \mathbf{0}$,
and $\{\eta_{k}^{(j)}\}_j$ is a sequence satisfying
$\eta_{k}^{(j)} \geq \hat{\eta}_k$ for $j\geq 1$,
and $\lim\limits_{j\rightarrow \infty} \eta_{k}^{(j)} = 1$. Based on these observations, next, we show that the repeat loop terminates after a finite number of steps.
\begin{theorem}~\label{thm:linesearch}
Let $\{\widehat{\Delta X}_k\}$ be the sequence generated from Algorithm~1, i.e.,
\begin{equation*}
\|(DF(X_k) [\widehat{\Delta X}_k] + F(X_k)\| \leq \eta_k \|F(X_k)\|.
\end{equation*}
Then, once $j$ is large enough, the sequence  $\{\eta_{k}^{(j)}\}_j$ satisfies the following:
 \begin{eqnarray}
 \|F(X_k) + DF(X_k) [\Delta{X}_k^{(j)}]\|
& \leq &
\eta_k^{(j)} \|F(X_k)\|,\nonumber\\
 \|F(R_{X_{k}}(
 \Delta X_k^{(j)}
 ) )\|
 &\leq&
 (1-t(1-\eta_k^{(j)}))\|F(X_k)\|. \label{eq:tem2}
 \end{eqnarray}

%
%
%


\end{theorem}

\begin{proof}
Let $\hat{\eta}_k$ be defined in Eq.~\eqref{eq:eta1}
with $\Delta X_k = \widehat{\Delta X}_k$,
and
$\epsilon_k = \frac{(1-t)(1-\hat{\eta}_k)\|F(X_k)\|}{\|\widehat{\Delta X}_k\|}$. Since $F$ is continuously differentiable, for $\epsilon_k > 0$, there exists a sufficiently small $\delta > 0$ such that $\|\Delta X\| < \delta$ implies that:
\begin{equation*}
\|
F(R_{X_{k}}(\Delta X) ) - F(R_{X_{k}}(\mathbf{0}_{X_{k}}))
-D F(R_{X_{k}}(\mathbf{0}_{X_{k}})) [\Delta X]
\| \leq \epsilon_k \|\Delta X\|,
\end{equation*}
where $\mathbf{0}_{X_{k}}$ is the origin of $T_{X_{k}}\left(\mathcal{O}(n) \times \mathcal{O}(n) \times \mathcal{W}(n)\right)$.

For $\delta > 0$, we let
\begin{equation*}
\eta_{\mathrm{min}} = \mathrm{max} \left \{\hat{\eta}_k, 1-\frac{(1-\hat{\eta}_k)\delta}{\|\widehat{\Delta X}_k\|}\right \}.
\end{equation*}
Note that once $j$ is sufficiently large,
\begin{equation}\label{eq:cond}
 \eta_k^{(j)} - \eta_{\mathrm{min}}
\geq\left (\frac{\delta}{\|\widehat{\Delta X}_k\|}
 -\Theta_{j-1}
\right )(1-\hat{\eta}_k)  {\geq} 0.
\end{equation}
For sufficiently large $j$, we consider the sequence $\{\Delta X_k^{(j)}\}_j$ in Eq.~\eqref{eq:Deltakj} with
 $\eta_k^{(j)} \in [\eta_{\mathrm{min}},1)$. We can see that:
 \begin{equation*}
 \|\Delta X_k^{(j)} \| =
 \|
 \frac{1- \eta_k^{(j)}}{1- \hat{\eta}_k}\widehat{\Delta X}_k
 \|\leq \frac{1-\eta_{\mathrm{min}}}{1- \hat{\eta}_k}\|\widehat{\Delta X}_k\| \leq \delta.
 \end{equation*}
 This implies that:
 \begin{eqnarray*}
 \|F(X_k) + DF(X_k) [\Delta{X}_k^{(j)}]\| &\leq&
 \left \|
 F(X_k) + DF(X_k) \left(\frac{1- \eta_k^{(j)}}{1- \hat{\eta}_k}\widehat{\Delta X}_k\right)
 \right \|\\
&\leq&
 \left \|
 \frac{ \eta_k^{(j)} -  \hat{\eta}_k}{1- \hat{\eta}_k}
 F(X_k) +
 \frac{1- \eta_k^{(j)}}{1- \hat{\eta}_k} \left(
 DF(X_k) [\widehat{\Delta X}_k] + F(X_k)\right)
 \right\|\\
&\leq&
\frac{ \eta_k^{(j)} -  \hat{\eta}_k}{1- \hat{\eta}_k}
\|
 F(X_k)\| +
 \frac{1- \eta_k^{(j)}}{1- \hat{\eta}_k}
 \hat{\eta}_k \|F(X_k)\|\\
& = &
\eta_k^{(j)} \|F(X_k)\|,
 \end{eqnarray*}
 and
 \begin{eqnarray*}
 F(R_{X_{k}}(
 \Delta X_k^{(j)}
 ) )\|
 &=&
 \|
F(R_{X_{k}}(
 \Delta X_k^{(j)}) - F(R_{X_{k}}(\mathbf{0}_{X_{k}}))
-D F(R_{X_{k}}(\mathbf{0}_{X_{k}})) [
 \Delta X_k^{(j)}]
\|\\
&&
+ \|F(X_k) + DF(X_k)
 [\Delta X_k^{(j)}]\|\\
 &=&
\epsilon_k \|
 \Delta X_k^{(j)}\| + \eta_k^{(j)}\|F(X_k)\|\\
 &=&
 \frac{(1-t)(1-\hat{\eta}_k)\|F(X_k)\|}{\|\widehat{\Delta X}_k\|}
\left\|  \frac{1- \eta_k^{(j)}}{1- \hat{\eta}_k}\widehat{\Delta X}_k \right\| + \eta_k^{(j)}\|F(X_k)\|\\
 &=&
 (1-t(1-\eta_k^{(j)}))\|F(X_k)\|.
 \end{eqnarray*}
\qed  \end{proof}

%
%
%
%
%
%
%
%
%
%
%
%
%
From the proof of Theorem~\ref{thm:linesearch}, we can see that
for each $k$, the repeat loop for the backtracking line search will terminate in a finite number of steps once condition~\eqref{eq:cond} is satisfied.
{Moreover, Theorem~\ref{lem:cov} and condition~\eqref{eq:steplengthnew} imply the following:
 \begin{equation*}
\lim\limits_{k\rightarrow \infty} \|\widehat{\Delta X}_k\| = \mathbf{0}.
\end{equation*}
That is, if $k$ is sufficient large, i.e., $\|\widehat{\Delta X}_k\|$ is small enough, then from the proof of Theorem~\ref{thm:linesearch}
we see that condition~\eqref{eq:line} is always satisfied, i.e., $\eta_k =\hat{\eta}_k$ for all sufficient large $k$. }

To show that Algorithm~1 is a globally convergent algorithm, we have one additional requirement for the retraction $R_X$, {i.e.}, there exist $\nu>0$ and $\delta_{\nu} > 0$ such that:
\begin{equation}\label{eq:retraction}
\nu\|\Delta X\| \geq  \mbox{dist}(R_X(\Delta X),X),
\end{equation}
for all $X \in \mathcal{O}(n) \times \mathcal{O}(n) \times \mathcal{W}(n)$ and for all
$\Delta X \in T_X\left (\mathcal{O}(n) \times \mathcal{O}(n) \times \mathcal{W}(n)\right)$ with
$\|\Delta X\|  \leq \delta_{\nu}$~\cite{Absil2008}. Here ``$\mbox{dist}(\cdot,\cdot)$" represents the Riemannian distance on  $ \mathcal{O}(n) \times \mathcal{O}(n) \times \mathcal{W}(n)$. Under this assumption, our next theorem shows the global convergence property of Algorithm~1.
We have borrowed the strategy for this proof from that used in~\cite[Theorem 3.5]{Eisenstat1994} to prove the nonlinear matrix equation. 

%

\begin{theorem}\label{thm:glbcov}
Assume that Algorithm~1 does not break down.
Let $X_*$ be a limit point of $\{X_k\}$. Then
$X_k$ converges to $X_*$ and $F(X_*) = \mathbf{0}$. Moreover, $X_k$ converges to $X_*$ quadratically whenever $X_k$ is sufficiently close to $X_*$.
\end{theorem}

\begin{proof}
Suppose $X_k$ does not converge to $X_*$. This implies that there exist two sequences of numbers $\{k_j\}$ and $\{\ell_j\}$ for which:
\begin{align*}
& X_{k_j} \in N_{{\delta}/{j}} (X_*), \\
& X_{k_j+\ell_j} \not\in N_{{\delta}} (X_*), \\
& X_{k_j + i} \in N_{\delta} (X_*), \mbox{ if } i = 1,\ldots, \ell_{j-1}\\
& k_j+\ell_j \leq k_{j+1}.
\end{align*}

From Theorem~\ref{thm:linesearch}, we see that
the repeat loop among {steps 6 to 9} of Algorithm~1 terminates in finite steps. For each $k$, let $m_k$ be the smallest number such that condition~\eqref{eq:tem2} is satisfied, i.e.,
$
\Delta X_k  = \Theta_{m_k} \widehat{\Delta X}_k $
{and
$\eta_k = 1- \Theta_{m_k}(1-\hat{\eta}_k)$}
with $\Theta_{m_k}$ being defined in Eq.~\eqref{Theta}. It follows from condition~\eqref{eq:length} that:
\begin{equation}\label{eq:DeltaXk2}
\|\Delta X_k\|
\leq
{ 2 \Theta_{m_k}
 \left (
 \frac{ 1+\eta_{\mathrm{max}} }{1-{\eta}_{\mathrm{max}}}
 \right )}
 (1-{\eta_{k}}) \|
DF(X_*)^\dagger
\|
 \|F(X_k)\|,
\end{equation}
for a sufficiently small $\delta$ and $X_k \in B_\delta (X_*)$, so that condition~\eqref{eq:dagger} is satisfied.  Let
\begin{equation*}
\Gamma_{m_k} =  {2 \Theta_{m_k} \left (
\frac{ 1+\eta_{\mathrm{max}} }{1-{\eta}_{\mathrm{max}}}
\right)}
 \|
DF(X_*)^\dagger
\|.
\end{equation*}
According to condition~\eqref{eq:retraction}, there exist
 $\nu>0$ and $\delta_{\nu} > 0$ such that:
\begin{equation*}
\nu\|\Delta X\| \geq  \mbox{dist}\left(R_X(\Delta X), X\right),
\end{equation*}
when $\|\Delta X\| \leq {\delta_{\nu}}$. Since $F(X_k)$ approaches  zero as $k$ approaches infinity,  for $\delta_{\nu}$, condition~\eqref{eq:DeltaXk2} implies that there exists a sufficiently large $k$ such that:
\begin{equation}\label{eq:DeltaXk3}
\nu\|\Delta X_k\| \geq  \mbox{dist}\left(R_{X_k}(\Delta X_k),X_k\right)
\end{equation}
is satisfied whenever $\|\Delta X_k\| \leq \delta_{\nu}$.

Then for a sufficiently large $j$, we can see from conditions~\eqref{eq:DeltaXk2} and \eqref{eq:DeltaXk3} that:
\begin{eqnarray*}
\frac{\delta}{2} &\leq & \mbox{dist}(X_{k_j+\ell_j},X_{k_j})
\leq \sum_{k=k_j}^{k_j+\ell_j-1}
\mbox{dist}(X_{k+1},X_{k})\\
& = & \sum_{k=k_j}^{k_j+\ell_j-1}
\mbox{dist}(R_{X_k}(\Delta X_{k}),X_{k})
\leq
\sum_{k=k_j}^{k_j+\ell_j-1}
\nu \|\Delta X_k\|\\
&\leq &
\sum_{k=k_j}^{k_j+\ell_j-1}
\nu \Gamma_{m_k}(1-\eta_k) \|F(X_k)\|
\leq
\sum_{k=k_j}^{k_j+\ell_j-1}
\frac{\nu \Gamma_{m_k}}{t}\left(
\|F(X_k)\| - \|F(X_{k+1})\|
\right)\\
&\leq& \frac{\nu \Gamma_{m_k}}{t}
\left ( \|
F(X_{k_j})\| - \|F(X_{k_{j + 1}})
\| \right).
\end{eqnarray*}
This is a contraction, since Theorem~\ref{lem:cov} implies that $F(X_{k_j})$ converges to zero as $j$ approaches infinity and $\Gamma_{m_k}$ is bounded. Thus,  $X_k$ converges to $X_*$, and immediately, we have $F(X_*) = \mathbf{0}$.   This completes the proof of the first part.

To show that  $X_k$ converges to $X_*$ quadratically once $X_k$ is sufficiently close to $X_*$, we let $C_1$ and $C_2$ be two numbers satisfying the following:
\begin{eqnarray*}
 \|
F(X_{k+1}) - F(X_{k})
-D F(X_k) [
 \Delta X_k]
\|
&\leq& C_1 \|
 \Delta X_k\|^2,\\
 \|
F(X_{k})
\|
&\leq& C_2 \mbox{dist} (X_k,X_*),
\end{eqnarray*}
for a sufficiently large $k$. {The above assumptions are true since $F$ is second differentiable and $F(X_*) = \mathbf{0}$}.
We can also observe that:
\begin{eqnarray}
\|
F(X_{k+1})
\| & \leq &
 \|
F(X_{k+1}) - F(X_{k})
-D F(X_k) [
 \Delta X_k]
\|
+ \|F(X_k) + DF(X_k)
 [\Delta X_k]\| \nonumber\\
 &\leq &
C_1 \|
 \Delta X_k\|^2 + \hat{\eta}_k \|F(X_k)\|
 \leq
 C_1
 (\Gamma_{m_k} \|F(X_k)\|)^2 +
 \|F(X_k)\|^2 \nonumber
 \\
 &\leq&
\left(C_1
 \Gamma^2 C_2^2+C_2^2\right) \mbox{dist} (X_k,X_*)^2, \label{eq:FXK2}
%
\end{eqnarray}
where $\Gamma =
{2  \left (
\dfrac{ 1+\eta_{\mathrm{max}} }{1-{\eta}_{\mathrm{max}}}
\right)}
 \|
DF(X_*)^\dagger
\|.
$

Since $X_k$ converges to $X_*$ as $k$ converges to infinity, for a sufficiently large $k$, it follows from conditions~\eqref{eq:DeltaXk2}, \eqref{eq:DeltaXk3},~\eqref{eq:tem2}, and~\eqref{eq:FXK2}
  that:
\begin{eqnarray*}
\mbox{dist}(X_{k+1}, X_*) &=& \lim_{p\rightarrow \infty} \mbox{dist}(X_{k+1}, X_p) \leq
\sum_{s = k}^\infty
\mbox{dist}\left(X_{s+1}, R_{X_{s+1}}(\Delta X_{s+1}) \right)  \\
&\leq& \sum_{s = k}^\infty
\nu \|\Delta X_{s+1}\|
\leq
\sum_{s=k}^{\infty}
\nu \Gamma_{m_{s+1}}(1+\eta_{\mathrm{max}}) \|F(X_{s+1})\|\\
&\leq&
\nu \Gamma
(1+\eta_{\mathrm{max}})
\sum_{j=0}^{\infty} (1-t(1-\eta_{\mathrm{max}}))^j\|F(X_{k+1})\|
\\
 &\leq& C \mbox{dist}(X_{k},X_*)^2,
\end{eqnarray*}
for some constant $C =  \dfrac{  \nu \Gamma
%
(1+\eta_{\mathrm{max}})
\left(C_1
 \Gamma^2C_2^2+C_2^2\right)
}{
 t(1-\eta_{\mathrm{max}})}
 $.
\qed
\end{proof}

It is true that we might assume without loss of generality that  the inverse of $DF(X_k)\circ DF(X_k)^*$ always exists numerically. However, 
once $DF(X_k)\circ DF(X_k)^*$ is ill-conditioned or (nearly) singular, we choose an operator $E_k = \sigma_k id_{T_{F(X_k)}}$, where $\sigma_k$ is a constant and $id_{T_{F(X_k)}}$ is an identity operator on $T_{F(X_k)} (\mathbb{R}^{n\times n})$ to make $DF(X_k)\circ DF(X_k)^* + \sigma_k id_{T_{F(X_k)}}$ well-conditioned or nonsingular. In the calculation, this replaces the calculation in Eq.~\eqref{eq:conjugate} with the following equation:
 \begin{equation*}
(DF(X_k)\circ DF(X_k)^* + \sigma_k id_{T_{F(X_k)}}) [\Delta Z_k] = -F(X_k).
\end{equation*}
That is, Algorithm~1 can be modified to fit in this case by replacing the satisfaction of condition~\eqref{eq:inm} with the following two conditions:
\begin{subequations}\label{eq:conjugate3}
 \begin{align}
&\|(DF(X_k)\circ DF(X_k)^* + \sigma_k id_{T_{F(X_k)}}) [\Delta {Z}_k]\| \leq \eta_{k} \|F(X_k)\|, \label{eq:conjugate31}\\
&
\|(DF(X_k)\circ DF(X_k)^*) [\Delta {Z}_k] + F(X_k)\| \leq \eta_{\mathrm{max}} \|F(X_k)\|,
\end{align}
\end{subequations}
where $\sigma_k :=  \mathrm{min}\left\{\sigma_{\mathrm{max}}, \|F(X_{k})\| \right\}$ is a selected
perturbation determined by the parameter $\sigma_{\mathrm{max}}$
and $\|F(X_{k})\|$.
Of course, we can provide the proof of the quadratic convergence under condition~\eqref{eq:conjugate3} without any difficulty (see~\cite{Zhao2017} for a similar discussion). Thus, we ignore the proof here. However, we note that even if a selected perturbation is applied to an ill-conditioned problem, the linear operator
$DF(X_k)\circ DF(X_k)^* + \sigma_k id_{T_{F(X_k)}}$ in condition~\eqref{eq:conjugate31} {might become nearly singular or ill-conditioned once $\sigma_k$ is small enough}. This will prevent the iteration in the CG method from converging in fewer than $n^2$ steps, and cause the value of $f'(0)$ to not be negative. This possibility suggests that we apply Algorithm~1 without performing any perturbation in our numerical experiments. If 
{the CG method cannot terminate within $n^2$ iterations}, it may be necessary to compute a new approximated solution ${\Delta Z}_k$ by selecting a new initial value for $X_0$.

\section{Numerical Experiments}
Note that the iteration of Algorithm~1 will be trapped without convergence to a solution if the IESP is unsolvable. As such, in our numerical experiments, we assume the existence of a solution of an IESP solution beforehand by generating sets of eigenvalues and singular values from a series of randomly generated matrices. 
{For a $2\times 2$ case, it is certain that Theorem~\ref{THM} in the appendix provides an alternative way to generate testing matrices.}
However, for general $n\times n$ matrices, the condition of the solvability of the IESP with some particular structure remains unknown and merits further investigation.
In this section, we
show how Algorithm~1
can be applied to solve an IESP with a particular structure.
We note that we performed all of the computations in this work in MATLAB version~2016a on a desktop  with a 4.2 GHZ Intel Core i7 processor and 32 GB of main memory. For our tests, we set
%
$\eta_{\mathrm{max}} = 0.9$, $\theta_{\mathrm{min}} = 0.1$, $\theta_{\mathrm{max}} = 0.9$, $t = 10^{-4}$, and $\epsilon = 10^{-10}$.  Also, in our computation, we emphasize two things. First, once the CG method computed in Algorithm~1 cannot be  terminated  within $n^2$ iterations, {restart Algorithm~1 with a different initial value $X_0$}. Second, {due to the rounding errors in numerical computation,} care must be taken in the selection of $\eta_{k}$ so that the upper bound
$\eta_k \|F(X_k)\|$ in condition~\eqref{eq:inm} {is not too small to cause the CG method abnormal. To this end, in our experiments, we use the condition
\begin{equation*}
\mbox{max}\{\eta_k \|F(X_k)\|, 10^{-12}\},
\end{equation*}
instead of $\eta_k \|F(X_k)\|$. The implementations of the Algorithm 1 are available online, say, 
~\url{http://myweb.ncku.edu.tw/~mhlin/Bitcodes.zip}. }

\begin{example} \label{ex1}
To demonstrate the capacity of our approach for solving problems that are relatively large, we randomly generate a set of eigenvalues and a set of singular values
of different size, say, $n = 20$, $60$, $100$, $150$, $200$, $500$, and {$700$} from matrices given by the MATLAB command:
\begin{equation*}
A = {\tt randn(n)}.
\end{equation*}
For each size, we perform $10$ experiments. To illustrate the elasticity of our approach, we randomly generate the initial value $X_0 = (U_0,V_0, W_0)$ in the following way:
\begin{equation*}
W_0 = {\tt   triu(randn(n))}, \, W_0{\tt (find(\Lambda))} = \mathbf{0},\mbox{ and } [ U_0,tmp,V_0 ]  = {\tt svd}(\Lambda+W_0).
\end{equation*}

In Table~\ref{tab5.1}, we show
the average residual value $({\rm Residual})$,
the average final error $({\rm Error})$, as defined by:
 \begin{equation*} 
{\mbox{final}\,\, \mbox{error}} =
\|\boldsymbol{\lambda}(A_{{\mathrm{new}}}) -  \boldsymbol{\lambda}\|_2 +
\|\boldsymbol{\sigma}(A_{{\mathrm{new}}}) -  \boldsymbol{\sigma}\|_2,
 \end{equation*}
the average  number of iterations within the {\rm CG} 
method {\rm (CGIt)}$\sharp$, {the average number of iterations within the inexact Newton method {\rm (INMIt)}$\sharp$, 
and the average elapsed time {\rm (Time)}},
as  performed by our algorithm. In Table~\ref{tab5.1}, we can see that {the elapsed time and the average number of iterations within the $\rm CG$ method increase dramatically} as the size of the matrices increases. 
This can be explained by the fact that the number of degrees of freedom of the problem increases significantly. Thus, the number of the iterations required by the {{\rm CG}} method {and the required computed time increase} correspondingly.  However, it is interesting to see that the required number of iterations {within the inexact Newton method} remains almost the same for matrices of different sizes. 
{One way to speed up the entire process of iterations is to transform the problem~\eqref{eq:conjugate}
into a form that is more suitable for the {\rm CG} method, for example, apply the {\rm CG} method with a preselected preconditioner. Still, this selection of the preconditioner requires further investigation.}

\begin{table}[h!!!]\label{tab5.1}
\begin{center}
\caption{Comparison of the required {\rm CGIt}$\sharp$, {\rm INMIt}$\sharp$, {\rm Residual}, {\rm Error} values, and Time for solving the IESP by Algorithm~1.}
\begin{tabular}
{ccccccccc}
\hline   $n$  & {CGIt$\sharp$} & {INMIt$\sharp$} 
& {Residual}
& \mbox{Error}
& \mbox{Time}
\\
\hline 
\rule{0pt}{2.3ex}%
20 & 208 & 9.4
& $5.54\times 10^{-12}$ 
&  $9.65 \times 10^{-13}$
&  $2.47 \times 10^{-2}$
   \\
%
60 & 740 & 10&
$8.13\times 10^{-12}$
&  $7.23 \times 10^{-13}$ 
&  $4.11 \times 10^{-1}$
\\
%
100 & 1231 & 10.4 &
$
1.06 \times 10^{-12}$
&  $ 
9.74
\times 10^{-14}$  
&  $2.22$
\\
%
150 & 1773 & 10.1 &
$1.01\times 10^{-12}$
&  $  1.06 \times 10^{-13}$  
& $6.82$
\\
%
200 & 1939 & 10.5 &
$1.20 \times 10^{-12}$
&  $ 1.49 \times 10^{-13}$  
& $19.3$
\\
%
500 & 6070 & 10.6 &
$1.47 \times 10^{-12}$
&  $ 4.12 \times 10^{-13}$  
& $665$
\\
%
700 & 8905 & 10.6 &
$5.42 \times 10^{-12}$
&  $ 7.24 \times 10^{-13}$  
& $2465$
\\
\hline
\end{tabular}
\end{center}
\end{table}

%
\end{example}

 \begin{example}\label{ex2}
 In this example, we use Algorithm~1 to construct a nonnegative matrix with prescribed eigenvalues and singular values and a specific structure. 
We specify this {\rm IESP} and call it the {\rm IESP} with desired entries {\rm (DIESP)}. The {\rm DIESP} can be defined as follows.

\vskip .1in
\begin{minipage}{0.9\textwidth}
{\rm (\textbf{DIESP})}
Given a subset $\mathcal{I} = \{(i_t,j_t)\}_{t=1}^\ell$ with double subscripts, a set of real numbers $\mathcal{K}  = \{k_t\}_{t=1}^\ell$,
a set of $n$ complex numbers $\{\lambda_i\}_{i=1}^n$, satisfying
  $\{\lambda_i\}_{i=1}^n = \{\bar{\lambda}_i\}_{i=1}^n$,
and a set of $n$ nonnegative numbers $\{\sigma_i\}_{i=1}^n$, find a nonnegative $n\times n$ matrix $A$ that has
eigenvalues $\lambda_1, \ldots, \lambda_n$,
singular values $\sigma_1,\ldots,\sigma_n$ and $A_{i_t,j_t} = k_t$ for $t = 1,\ldots, \ell$.

\end{minipage}
\vskip .1in

Note that once $i_t= j_t = t$ for $t = 1,\ldots, n$, we investigate a numerical approach for solving the {\rm IESP} with prescribed diagonal entries. As far as we know, {the research result close to} this problem is only available in~\cite{Chu2017}. However, for a general structure, no research has been conducted to implement this investigation.
To solve the {\rm DIESP}, our first step is to obtain a real matrix $A$ with prescribed eigenvalues and singular values. Our second step is to derive entries of $Q^\top A Q$, where $Q\in\mathcal{O}(n)$, that satisfy the nonnegative property and desired values determined by the sets $\mathcal{I}$ and $\mathcal{K}$.
We solve the first step in the same manner as in Example~\ref{ex1}, but for the second step, we consider the following two sets $\mathcal{L}_1$ and $\mathcal{L}_2$, which are defined by:
\begin{eqnarray*}
\mathcal{L}_1 &=&  \{A \in\mathbb{R}^{m\times n} \,|\,
A_{i_t,j_t} = k_t,\,\, \mbox{for}\,\, 1\leq t\leq \ell; \mbox{otherwise }
A_{i,j} = 0
\}, \\
\mathcal{L}_2 &=& \{A \in\mathbb{R}^{m\times n} \,|\,
A_{i,j} = 0,\,\, \mbox{for}\,\,
1\leq i,j \leq n \,\, \mbox{and} \,\, (i,j)\in\mathcal{I}\},
\end{eqnarray*}
and then solve the following problem:
\begin{equation}\label{eq:fxs}
\mbox{find }
P \in \mathcal{L}_2 \mbox{ and }
Q \in
\mathcal{O}(n)
\mbox{ such that }
H(P,Q) =\hat{A} + P\odot P - Q AQ^\top = \mathbf{0},
%
\end{equation}
%
with $\hat{A} \in\mathcal{L}_1$.
 Let
$ [A, B]:= AB - BA$ denote the Lie bracket notation.
It follows from direct computation that the corresponding differential $DH$ and its adjoint $DH^*$ have the following form~\cite{Zhao2017}:
\begin{eqnarray*}
 DH(P,Q)[(\Delta P, \Delta Q)] & = &
2 P\odot \Delta P +  [  Q A Q^\top,
    \Delta Q Q^\top],\\
 DH(P,Q)^*
 [\Delta Z] & = &
 \left (
2P\odot \Delta Z,
 \frac{1}{2}(
[QA Q^\top, \Delta Z^\top]
+
[Q A^\top Q^\top, \Delta Z]
)Q
 \right ),
\end{eqnarray*}
and, for all $(\xi_P, \xi_Q) \in T_{(P,Q)} (\mathcal{L}_2
 \times \mathcal{O}(n))$, we can compute
the retraction $R$ using the following formula:
\begin{equation*}
R(P,Q) = (R_P(\xi_P),R_Q(\xi_Q)),
\end{equation*}
where
\begin{equation*}
R_P(\xi_P) = P+\xi_P, \quad
R_Q(\xi_Q) = qf(Q+\xi_Q).
\end{equation*}

For these experiments, we randomly generate nonnegative matrices $20\times 20$ in size by the MATLAB command ``$A = rand(20)$"
to provide the desired eigenvalues, singular values, and diagonal entries, i.e., to solve the {\rm DIESP} with the specified diagonal entries.
We record the~{final error}, as given by the following formula:
 \begin{equation*} 
\mbox{final error} =
\|\boldsymbol{\lambda}(A_{\mathrm{new}}) -  \boldsymbol{\lambda}\|_2 +
\|\boldsymbol{\sigma}(A_{\mathrm{new}}) -  \boldsymbol{\sigma}\|_2
+ \|(A_{\mathrm{new}})_{i_t,j_t } - k_t \|_2.
 \end{equation*}
 After randomly choosing 10 different matrices, Table~\ref{tab5.2} shows our results with the intervals {\rm (Interval)} containing all of the residual values and final errors, and their corresponding average values {\rm (Average)}.  These results provide sufficient evidence that Algorithm~1 can be applied to solve the DIESP with high accuracy.

%
%
%
%
%
 \begin{table}[h!!!]\label{tab5.2}
\begin{center}
\caption{Records of final errors and residual values for solving the DIESP by Algorithm~1.}
\begin{tabular}{ l c c }
  \hline   & Interval   & Average  \\\hline
  \rule{0pt}{2.3ex}%
  final errors & $[7.27\times 10^{-13}, 1.21\times 10^{-11}]$
& $2.91\times 10^{-12}$ \\
  residual values  & $[7.77\times 10^{-13}, 4.93 \times 10^{-12}]$  & $1.85\times 10^{-12}$  \\\hline
\end{tabular}
\end{center}
\end{table}

%
%
%
%
%

\end{example}

Although  Example~\ref{ex2} considers examples with a  nonnegative structure, we emphasize that
Algorithm~1 can work with entries that are not limited to being nonnegative. That is,
%
%
to solve the IESP without nonnegative constraints but with another specific structure, Algorithm~1
can fit perfectly well by replacing $H(P,Q)$ in problem~\eqref{eq:fxs}  with
\begin{equation*}
G(S,Q) :=  \hat{A} + S - QAQ^\top,
\end{equation*}
where $\hat{A} \in\mathcal{L}_1$,  $S\in\mathcal{L}_2$ and $Q\in\mathcal{O}(n)$.

\section{Conclusions}
{
In this paper, we apply the Riemannian inexact Newton method to solve an initially complicated and challenging IESP. We provide a thorough analysis of the entire iterative processes and show that this algorithm converges globally and quadratically to the desired solution. We must emphasize that our theoretical discussion and numerical implementations can also be extended to solve an IESP with a particular structure such as desired diagonal entries and a matrix whose entries are nonnegative. This capacity can be observed in our numerical experiments. It should be emphasized that this research is the first to provide a unified and effective means to solve the IESP with or without a particular structure.

However, the numerical stability for extremely ill-conditioned problems is a case that we should pay attention to, though reselecting the initial values could be a strategy to get rid of this difficulty.  Another way to tackle this difficulty is to select a good preconditioner. But, the operator encountered in our algorithm is nonlinear and high-dimensional. {Thus, the selection of the preconditioner could involve the study of tensor analysis, where further research is needed.}

Theoretically determining the sufficient and necessary condition for solving IESPs of any specific structure, including a stochastic, Toeplitz, or Hankel structure, is challenging and interesting. In the appendix, we provide the solvability condition of the IESP with real or nonnegative matrices of size $2\times 2$ real/nonnegative matrices, while the desired eigenvalues, singular values, and main diagonal entries are given. We hope that this discussion can motivate a further discussion shortly.

}

%
%
%
%
%
%

\appendix
\section{Appendix
}

\subsection{The solvability of the IESP of a $2\times 2$ matrix}

For the IESP, the authors in~\cite{Chu2017} use a geometric argument to investigate a necessary and sufficient condition for the  existence of a $2 \times 2$ real matrix with prescribed diagonal entries. This argument also leads to a sufficient algebraic but not necessary condition 
for the construction of a $2\times 2$ real matrix. In this {appendix}, the algebraic condition under which a $2 \times 2$ real matrix or even nonnegative matrix can be constructed in closed form, given its eigenvalue, singular values, and main diagonal entries.
To do so, we must have the following results. The first result, the so-called Mirsky condition, provides
the classical relationship between
the eigenvalues $\boldsymbol{\lambda} = \{\lambda_1,\ldots,\lambda_n\}$
 and
the diagonal entries $\mathbf{d}
= \{d_1,\ldots,d_n\}
$.

\begin{theorem}{\rm{[\cite{Mirsky1958}, Mirsky condition].}}\label{Mirsky}
There exists a real matrix $A\in\mathbb{R}^{n\times n}$ having
eigenvalues $\boldsymbol{\lambda} = \{\lambda_1,\ldots,\lambda_n\}$ and
main diagonal entries $\mathbf{d}
= \{d_1,\ldots,d_n\}
$, that are possibly in different order, if and only if
\begin{equation}\label{eq:Mirsky}
\sum_{i=1}^n \lambda_i = \sum_{i=1}^n d_i.
\end{equation}
\end{theorem}

The second result provides the relationship between the singular values $\boldsymbol{\sigma}$ and main diagonal entries $\boldsymbol{d}$ of a $2\times 2$ nonnegative matrix.
\begin{theorem}
{\rm{[\cite{Wu2014}, Theorem 2.1].}}
%
\label{THM2}
There exists a nonnegative matrix
$
A=
\left[\begin{array}{cc}d_1 & b \\c & d_2\end{array}\right]
\in\mathbb{R}^{2\times 2}
$
having  the singular values $\sigma_1\geq\sigma_2$ and main diagonal entries $d_1\geq d_2$, with renumbering if necessary,  if and only if
\begin{subequations}\label{eq:nisvp}
\begin{eqnarray}
  \sigma_1 + \sigma_2 \geq d_1+d_2,\,\,
  \sigma_1 - \sigma_2 &\geq& d_1-d_2,\quad
     \mbox{if }
 bc - d_1d_2 \leq 0,\label{eq:22cond}\\
\sigma_1- \sigma_2 &\geq& d_1 + d_2, \quad  \mbox{if }bc - d_1d_2 > 0.\label{eq:22cond2}
\end{eqnarray}
\end{subequations}

\end{theorem}

In particular, entries from matrix $A$ can be relaxed to real numbers, and condition~\eqref{eq:nisvp} is also true for the construction of a $2\times 2$ real matrix. The proof is almost identical to that in~\cite[Lemma 2.1]{Wu2014}. The major change is the substitution of nonnegative entries for real entries. Thus, we skip its proof here.
%
\begin{theorem}\label{THM}
There exists a real matrix
$
A=
\left[\begin{array}{cc}d_1 & b \\c & d_2\end{array}\right]
\in\mathbb{R}^{2\times 2}
$
having  singular values $\sigma_1\geq\sigma_2$ and main diagonal entries $d_1\geq d_2$, with renumbering if necessary,  if and only if
\begin{subequations}
\begin{eqnarray*}
  \sigma_1 + \sigma_2 \geq d_1+d_2,\,\,
  \sigma_1 - \sigma_2 &\geq& d_1-d_2,\quad
     \mbox{if }
 bc - d_1d_2 \leq 0,
 \\
\sigma_1- \sigma_2 &\geq& d_1 + d_2, \quad  \mbox{if }bc - d_1d_2 > 0.
\end{eqnarray*}
\end{subequations}

\end{theorem}

Now we have the condition of the existence of a $2\times 2$ matrix provided with eigenvalues and main diagonal entries, or singular values and main diagonal entries. The next theorem, unsolved in~\cite{Wu2014}, deals with the case in which the three constraints---eigenvalues, singular values, and
main diagonal entries---
are of simultaneous concern.

\begin{theorem}\label{thm:three}
There exists a real matrix
$
A=
\left[\begin{array}{cc}d_1 & b \\c & d_2\end{array}\right]
\in\mathbb{R}^{2\times 2}
$
having {eigenvalues $|\lambda_1|\geq |\lambda_2|$,
singular values $\sigma_1\geq\sigma_2$, and
 main diagonal entries $d_1\geq d_2$,}
with renumbering if necessary,
 if and only if
 \begin{equation}\label{eq:eigdiasig}
 \lambda_1+\lambda_2  = d_1+d_2, \,\,
\sigma_1\geq |\lambda_1|,\,\, |\lambda_1\lambda_2| = \sigma_1\sigma_2,
 \end{equation}
 and
\begin{subequations}\label{eq:diasig}
\begin{eqnarray}
  \sigma_1 + \sigma_2 \geq d_1+d_2,\,\,
  \sigma_1 - \sigma_2 &\geq& d_1-d_2,\quad
     \mbox{if }
 bc - d_1d_2 \leq 0,\\
\sigma_1- \sigma_2 &\geq& d_1 + d_2, \quad  \mbox{if }bc - d_1d_2 > 0.
\end{eqnarray}
\end{subequations}

 \end{theorem}

\begin{proof}
Assume that conditions~\eqref{eq:eigdiasig} and~\eqref{eq:diasig} are satisfied.
Following from the Weyl-Horn and Mirsky conditions, we know that for any $2\times 2$ matrix, its eigenvalues, singular values, and diagonal entries must satisfy condition~\eqref{eq:eigdiasig}. Thus, Theorem~\ref{THM} implies that once condition~\eqref{eq:diasig} is satisfied, it suffices to say that there exists a $2\times 2$ real matrix.
%
%
%
%

On the other hand, the sufficient condition follows directly from the Weyl-Horn condition~\eqref{eq:Weyl}, the Mirsky condition~\eqref{eq:Mirsky}, and Theorem~\ref{THM}. This completes the proof.
\qed

\end{proof}

Since the solvability conditions of Theorem~\ref{THM2} and Theorem~\ref{THM} are equivalent, we can see that the solvability condition in Theorem~\ref{thm:three} can be confined to be the necessary and sufficient condition for the existence of a nonnegative $2\times 2$ matrix. {We summarize this result as follows.}
\begin{corollary}
There exists a nonnegative matrix
$
A=
\left[\begin{array}{cc}d_1 & b \\c & d_2\end{array}\right]
\in\mathbb{R}^{2\times 2}
$
having {eigenvalues $|\lambda_1|\geq |\lambda_2|$,
singular values $\sigma_1\geq\sigma_2$, and
 main diagonal entries $d_1\geq d_2$,}
with renumbering if necessary,
 if and only if
 \begin{equation*}
 \lambda_1+\lambda_2  = d_1+d_2, \,\,
\sigma_1\geq |\lambda_1|,\,\, |\lambda_1\lambda_2| = \sigma_1\sigma_2,
 \end{equation*}
 and
\begin{subequations}
\begin{eqnarray*}
  \sigma_1 + \sigma_2 \geq d_1+d_2,\,\,
  \sigma_1 - \sigma_2 &\geq& d_1-d_2,\quad
     \mbox{if }
 bc - d_1d_2 \leq 0,\\
\sigma_1- \sigma_2 &\geq& d_1 + d_2, \quad  \mbox{if }bc - d_1d_2 > 0.
\end{eqnarray*}
\end{subequations}

 \end{corollary}

Note that conditions~\eqref{eq:eigdiasig} and~\eqref{eq:diasig}  cannot be directly generalized to higher dimensional cases.
The authors in~\cite{Wu2014} present the necessary and sufficient condition of the existence of a real matrix with a size greater than $2$ and having prescribed eigenvalues, singular values, and main diagonal entries. However, given eigenvalues, singular values, and main diagonal entries,
no study has yet demonstrated the construction of a nonnegative matrix with a size greater than $2\times 2$.  This difficulty can be tackled by the use of {our} numerical computations.

%
%
%
%

  \section*{Acknowledgment}
The authors wish to thank Prof. Michiel E. Hochstenbach for 
his highly valuable comments. They also thank Prof. Zheng-Jian Bai and Dr. Zhi Zhao for helpful discussions.

\bibliographystyle{abbrv}

\end{document}